\documentclass[12pt]{article}

\usepackage{mathtools}
\usepackage{amssymb}
\usepackage{amsthm}
\usepackage{xcolor}
\usepackage{tikz}
\usepackage{fullpage}
\usepackage{url}

\newcommand{\red}[1]{\textcolor{black}{#1}}

\allowdisplaybreaks

\DeclareSymbolFont{bbold}{U}{bbold}{m}{n}
\DeclareSymbolFontAlphabet{\mathbbold}{bbold}

\newcommand{\N}{\mathbb{N}}
\newcommand{\Z}{\mathbb{Z}}

\newcommand{\R}{\mathbb{R}}
\newcommand{\C}{\mathbb{C}}

\renewcommand{\epsilon}{\varepsilon}

\renewcommand{\i}{{\rm i}}

\DeclareMathAccent{\Circ}{\mathalpha}{operators}{"17}



\DeclareMathOperator{\spt}{spt}

\DeclareMathOperator{\dive}{div}
\DeclareMathOperator{\curl}{curl}
\DeclareMathOperator{\grad}{grad}
\DeclareMathOperator{\diag}{diag}

\DeclareMathOperator{\kar}{ker}
\DeclareMathOperator{\rge}{ran}
\DeclareMathOperator{\dom}{dom}

\DeclareMathOperator{\m}{m}
\renewcommand{\Re}{\operatorname{Re}}

\let\phi\varphi

\let\leq\leqslant

\let\geq\geqslant

\makeatletter
\def\@row#1,{#1\@ifnextchar;{\@gobble}{&\@row}}
\def\@matrix{%
    \expandafter\@row\my@arg,;%
    \@ifnextchar({\\ \get@in@paren{\@matrix}}{\after@matrix}%
    }
\def\matrixtype#1#2#3{%
    \ifmmode\def\after@matrix{\end{#2}\right#3}%
    \else\def\after@matrix{\end{#2}\right#3$}$\fi
    \left#1\begin{#2}\get@in@paren{\@matrix}%
    }
\def\@column#1,{#1\@ifnextchar;{\@gobble}{\\ \@column}}
\newcommand\vect{}
\def\svect(#1){\left(\begin{smallmatrix}\@column#1,;\end{smallmatrix}\right)}
\def\vect{\get@in@paren{\@vect}}
\def\@vect{\left(\begin{matrix}\expandafter\@column\my@arg,;\end{matrix}\right)}
\def\get@in@paren#1({\def\my@arg{}\def\my@rest{}\def\after@get{#1}\get@arg}
\let\e@a\expandafter
\def\get@arg#1){\e@a\kl@test\my@rest#1(;}
\def\kl@test#1(#2;{\e@a\def\e@a\my@arg\e@a{\my@arg#1}%
                   \ifx:#2:\let\my@exec\after@get
                   \else\let\my@exec\get@arg
                        \e@a\def\e@a\my@arg\e@a{\my@arg(}%
                        \def@rest#2;%
                   \fi\my@exec}
\def\def@rest#1(;{\def\my@rest{#1\kl@zu}}
\def\kl@zu{)}

\makeatletter
\newcommand\MyPairedDelimiter{%
  \@ifstar{\My@Paired@Delimiter{{}}}
          {\My@Paired@Delimiter{}}%
}
\newcommand\My@Paired@Delimiter[4]{%
  \newcommand#2{%
    \@ifstar{\start@PD{#1}{\delimitershortfall=-1sp}{#3}{#4}}
            {\start@PD{#1}{}{#3}{#4}}%
  }%
}
\newcommand\start@PD[5]{%
  #1\mathopen{\mathpalette\put@delim@helper{\put@delim{#2}{#3}{.}{#5}}}%
  #5%
  \mathclose{\mathpalette\put@delim@helper{\put@delim{#2}{.}{#4}{#5}}}%
}
\newcommand\put@delim@helper[2]{%
  \hbox{$\m@th\nulldelimiterspace=0pt #2#1$}%
}
\newcommand\put@delim[5]{%
  \setbox\z@\hbox{$\m@th#5{#4}$}%
  \setbox\tw@\null
  \ht\tw@\ht\z@ \dp\tw@\dp\z@
  #1#5%
  \left#2\box\tw@\right#3%
}

\makeatother
\MyPairedDelimiter*{\abs}{\lvert}{\rvert}
\MyPairedDelimiter*{\norm}{\lVert}{\rVert}
\MyPairedDelimiter{\set}{\{}{\}}

\theoremstyle{plain} 
\newtheorem{theorem}{Theorem}[section]
\newtheorem{corollary}[theorem]{Corollary}
\newtheorem{lemma}[theorem]{Lemma}
\newtheorem{proposition}[theorem]{Proposition}
\theoremstyle{definition}
\newtheorem{example}[theorem]{Example}

\newtheorem*{definition}{Definition}
\newtheorem{remark}[theorem]{Remark}

\usepackage{enumitem}

\setenumerate[1]{nolistsep} 
\setenumerate[2]{nolistsep} 

\begin{document}

\title{Nonlocal homogenisation theory for curl-div-systems}

\author{Serge Nicaise and Marcus Waurick}

\date{}
 
\maketitle

\begin{abstract}We study the curl-div-system with variable coefficients and a nonlocal homogenisation problem associated with it. Using, in part refining, techniques from nonlocal $H$-convergence for closed Hilbert complexes, we define the appropriate topology for possibly nonlocal and non-periodic coefficients in curl-div systems to model highly oscillatory behaviour of the coefficients on small scales. We address curl-div systems under various boundary conditions and analyse the limit of the ratio of small scale over large scale tending to zero. Already for standard Dirichlet boundary conditions and local coefficients the limit system is nontrivial and unexpected. Furthermore, we provide an analysis of highly oscillatory local coefficients for a curl-div system with impedance type boundary conditions relevant in scattering theory for Maxwell's equations and relate the abstract findings to local $H$-convergence and weak$*$-convergence of the coefficients.
\end{abstract}



\medmuskip=4mu plus 2mu minus 3mu
\thickmuskip=5mu plus 3mu minus 1mu
\belowdisplayshortskip=9pt plus 3pt minus 5pt

\section{Introduction}

Homogenisation theory is concerned with the asymptotic behaviour of heterogenous materials when the ratio of the microscopic over macroscopic scale tends to 0. There is a vast literature concerning this and related questions, let us just mention the classical references \cite{Bensoussan1978,CioDon,Jikov1994,Tar09}.
Originally the notion of local $H$-convergence was introduced and applied to standard problems, like the
$\dive(a\grad )$ system where $a$ is a $L^\infty$ matrix function, see for instance \cite{Murat1997,Tar09}. Recently motivated by new physical  applications, like meta-materials, the notion of nonlocal $H$-convergence was introduced by the second author \cite{W18}, where some applications to $\dive(a\grad )$ or $\curl(a\curl )$ systems  are given.

In the present study, we focus on the curl-div-system relevant in the study of Maxwell's equations and scattering problems related to it, see \cite{ACL18,GLCI,Monk03,Nedelec00}.

More precisely, let $\Omega\subseteq \R^3$ be an open and bounded subset with weak Lipschitz boundary (i.e.~a Lipschitz submanifold of $\R^3$). Furthermore, we shall assume that the spaces of homogeneous Neumann and Dirichlet fields are both zero-dimensional.

Let $a\in \mathcal{B}(L^2(\Omega)^3)$ and $b\in \mathcal{B}(L^2(\Omega))$.  To set the stage, we address finding $u\in L^2(\Omega)^3$ such that for a suitable class of test functions $\phi, \psi$ we have
\[
       \langle a\curl u, \curl \phi \rangle + \langle b\dive u, \dive \psi\rangle = f(\phi)+g(\psi),
\]
where $u\in H_0(\curl)\cap H(\dive)$. The right-hand sides $f$ and $g$ are living in appropriate distribution spaces. 

We shall show below that under suitable conditions on positive definiteness for $a$ and $b$ that the problem to find $u\in H_0(\curl)\cap H(\dive)$ for given $f$ and $g$ such that the above variational equality holds is well-posed. The main question we address in this manuscript is the continuous dependence of the solution operator $S_{a,b}\colon (f,g)\mapsto u$ on the operators $a$ and $b$. More precisely, we shall show that $(a,b)\mapsto S_{a,b}$ is continuous if we endow (a subspace of) $\mathcal{B}(L^2(\Omega)^3)\times \mathcal{B}(L^2(\Omega))$ with the topology induced by nonlocal $H$-convergence and the target space with the weak operator topology. This result forms a new application of nonlocal $H$-convergence, a concept recently introduced in \cite{W18}. Moreover, we emphasise that even for local operators, that is, operators induced by multiplication with $L^\infty(\Omega)$-matrix fields, this result is new and surprising.

Indeed, given a bounded, measurable $[0,1)^3$-periodic function $d\colon \R^3\to \R^3$ and consider $d_n\coloneqq d(n\cdot)$ for all $n\in \N$ as multiplication operator in $\mathcal{B}(L^2(\Omega)^3)$ and $\mathcal{B}(L^2(\Omega))$. Let $u_n\in H_0(\curl)\cap H(\dive)$ be the solution of the variational problem
\begin{equation}\label{serge:13/3:1}
        \langle d_n\curl u_n, \curl \phi \rangle + \langle d_n\dive u_n, \dive \psi\rangle = f(\phi)+g(\psi).
\end{equation}
Then $(u_n)_n$ converges weakly to $u$, which is the unique solution of 
\[
          \langle e_{\textnormal{hom}} \curl u, \curl \phi \rangle + \langle M(d)\dive u, \dive \psi\rangle = f(\phi)+g(\psi),
\]
where $M(d)\coloneqq (\int_{[0,1)^3} \frac1d)^{-1}$ and $e_{\textnormal{hom}}$ is the inverse of the homogenised matrix associated to the sequence of coefficients 
\[
    \begin{pmatrix} d_n^{-1} & 0 & 0 \\ 0 & d_n^{-1} & 0 \\ 0 & 0 & d_n^{-1} \end{pmatrix}.
\]Note that $e_{\textnormal{hom}}$ is in general \emph{not} representable as a scalar multiplication operator anymore. In fact, even in the case of layered materials, $e_{\textnormal{hom}}$ is composed of both $M(d)^{-1}$ and $M(1/d)$ in the main diagonal, see Theorem \ref{thm:homloc} below.

To perform our analysis, we actually introduce a notion of  abstract curl-div systems and prove some convergence results within this abstract framework. These results are then applied to local and non-local systems of the form \eqref{serge:13/3:1}.

We also address similar questions for impedance boundary conditions for which the previous framework cannot be used. Hence an ad-hoc convergence theorem is provided.

The paper is organised as follows: in Section \ref{snlcge} we  recall the notion of nonlocal $H$-convergence. The
 abstract curl-div system is introduced in Section \ref{s:abstractcurl-div}
and its well-posedness  is proved.
Section \ref{acgence} is   devoted to some  abstract convergence results. Applications of these results to 
 local and non-local systems
  are given in 
Sections \ref{s:applilocal} and \ref{s:applinonlocal} respectively.
Section  \ref{s:impbc} where the case of impedance boundary conditions is considered concludes the paper.

For $s\geq 0$, $H^{s}(\Omega)$ denotes the standard Sobolev space in $\Omega$ of order $s$
equipped with its natural norm. The spaces
$H(\dive)$ and $H(\curl)$ are defined by
\begin{eqnarray*}
H(\dive):=\{u\in L^2(\Omega)^3: \dive u\in L^2(\Omega)\},\\
H(\curl):=\{u\in L^2(\Omega)^3: \curl u\in L^2(\Omega)^3\},
\end{eqnarray*}
also equipped with their natural norms.  Furthermore, we denote by $H_0(\curl)$ the completion of $C_c^\infty(\Omega)$-vector fields under the norm of $H(\curl)$.
Finally, we note that all abstract Hilbert spaces considered here are assumed to be separable.

\section{Nonlocal $H$-convergence\label{snlcge}}

In this section, we shortly recall the notion of nonlocal $H$-convergence as introduced in \cite{W18}. 

Let $H_0,H_1,H_2$ be Hilbert spaces and $A_0\colon \dom(A_0)\subseteq H_0 \to H_1$ and $A_1\colon \dom(A_1)\subseteq H_1 \to H_2$  be densely defined, closed linear operators with closed ranges satisfying the following property $\rge(A_0)= \kar(A_1)$. 

In order to capture the above problem class we are mostly concerned with the following particular application:
\begin{example}\label{ex:1} Let $\Omega\subseteq \mathbb{R}^3$ be an open bounded weak Lipschitz domain with no harmonic Dirichlet fields, see \cite{P82} for the corresponding geometric characterisation of $\Omega$ having connected complement. We refer to \cite{PW20} for a more recent source and generalisations heavily using the notion of Hilbert complexes.

  (a) $H_0= L^2(\Omega)$, $H_1=L^2(\Omega)^3$, $H_2=L^2(\Omega)$, $A_0=\grad$ with $\dom(A_0)=H_0^1(\Omega)$, $A_1=\curl$ with $\dom(A_1)=H_0(\curl)$.
  
  (b) $H_0=L^2(\Omega)^3$, $H_1=L^2(\Omega)$, $H_2=\{0\}$, $A_0=\dive$ with $\dom(A_0)= H(\dive)$, $A_1=0$. 
\end{example}

For the definition of nonlocal $H$-convergence, we need to introduce some additional operators associated with $A_0$ and $A_1$:

We define $\mathcal{A}_0 \colon \dom(A_0)\cap \kar(A_0)^\bot \to \rge(A_0), \phi \mapsto A_0\phi$ and, similarly, $\mathcal{A}_1^*\colon \dom(A_1^*)\cap \kar(A_1^*)^\bot\to \rge(A_1^*),\psi\mapsto A_1^*\psi$. Furthermore, we set $\dom(\mathcal{A}_0)$ to be the set $\dom(A_0)\cap \kar(A_0)^\bot$ endowed with the graph norm of $A_0$, which makes it a Hilbert space (and similarly for $\dom(\mathcal{A}_1^*)$). Note that due to the closed graph theorem, both $\mathcal{A}_0$ and $\mathcal{A}_1^*$ acting as $A_0$ and $A_1^*$ are topological isomorphisms. Due to the orthogonal decompositions $H_1=\rge(A_0)\oplus \kar(A_0^*)=\rge(A_0)\oplus \rge(A_1^*)$, we may represent $a\in \mathcal{B}(H_1)$ as block operator matrix
\[
    \begin{pmatrix} a_{00} & a_{01} \\ a_{10} & a_{11} \end{pmatrix} \in \mathcal{B}(\rge(A_0)\oplus \rge(A_1^*)).
\]
We let for $0<\alpha\leq\beta$
\begin{multline*}
   \mathcal{M}(\alpha,\beta, (A_0,A_1))\coloneqq \{ a\in \mathcal{B}(H_1);  \Re a_{00}\geq \alpha, \Re a_{00}^{-1}\geq 1/\beta,\\ a \text{ invertible}, \Re (a^{-1})_{11}\geq 1/\beta, \Re  (a^{-1})_{11}^{-1}\geq \alpha\},
\end{multline*} which defines the class of admissible coefficients.
\begin{remark}
Note that for all  $a\in \mathcal{B}(H_1)$ with the property $\Re a\geq \gamma>0$ there exists $0<\alpha\leq\beta$ such that $a\in \mathcal{M}(\alpha,\beta,(A_0,A_1))$.
\end{remark}
The definition of nonlocal $H$-convergence is now given as follows.

\begin{definition}
  Let $(a_n)_n$ be a sequence in $\mathcal{M}(\alpha,\beta,(A_0,A_1))$ and $a\in \mathcal{M}(\alpha,\beta,(A_0,A_1))$. Then $(a_n)_n$ \emph{nonlocally $H$-converges w.r.t. $(A_0,A_1)$ to} $a$, if
  for all $f\in \dom(\mathcal{A}_0)^*$ and $g\in \dom(\mathcal{A}_1^*)^*$ we have that $u_n\in \dom(\mathcal{A}_0)$ and $v_n\in \dom(\mathcal{A}_1^*)$ being the unique solutions of
  \[ 
      \langle a_n A_0 u_n, A_0  \phi\rangle = f(\phi)\quad \langle a_n^{-1} A_1^* v_n, A_1^* \psi \rangle = g(\psi)
  \]
  for all $\phi\in \dom(\mathcal{A}_0)$ and $\psi\in \dom(\mathcal{A}_1^*)$ weakly converge to $u\in \dom(\mathcal{A}_0)$ and $v\in \dom(\mathcal{A}_1^*)$. Moreover, $a_nA_0 u_n \rightharpoonup a A_0 u$ and $a_n^{-1}A_1^*v_n \rightharpoonup a^{-1}A_1^*v$, where $u$ and $v$ uniquely solve
  \[ 
      \langle a A_0 u, A_0  \phi\rangle = f(\phi)\quad \langle a^{-1} A_1^* v, A_1^* \psi \rangle = g(\psi)
  \]
    for all $\phi\in \dom(\mathcal{A}_0)$ and $\psi\in \dom(\mathcal{A}_1^*)$.
\end{definition}

Without further reference, we shall use the `sub-sequence-princpile' for nonlocal $H$-convergence. This, however, can only be applied, if nonlocal $H$-convergence induces a topological space. For this, we recall one of the main theorems in \cite{W18}.

\begin{theorem}[{{\cite[Theorem 5.3]{W18}}}]\label{thm:NHC_top} Let $\tau_{\textnormal{nlh}}$ be the topology on $\mathcal{M}(\alpha,\beta,(A_0,A_1))$ induced by the continuity mappings
\[
  a\mapsto a_{00}^{-1},\ a\mapsto a_{10}a_{00}^{-1},\ a\mapsto a_{00}^{-1}a_{01},\ a\mapsto a_{11}-a_{10}a_{00}^{-1}a_{01},
\]where the operator spaces in the co-domains are all endowed with the weak operator topology.

Then $(a_n)_n$ nonlocally $H$-converges w.r.t.~$(A_0,A_1)$ to some $a$, if and only if $(a_n)_n\to a$ in $\big(\mathcal{M}(\alpha,\beta,(A_0,A_1)),\tau_{\textnormal{nlh}}\big)$.
\end{theorem}

\section{Solution theory for abstract curl-div-systems\label{s:abstractcurl-div}}

As before, we let $A_0$ and $A_1$ be densely defined, closed linear operators with closed ranges satisfying the exact complex property $\rge(A_0)=\kar(A_1)$. Additionally, we let $A_2\colon \dom(A_2)\subseteq H_2\to H_3$ and $A_3\colon \dom(A_3)\subseteq H_3\to H_4$ be densely defined, closed and linear with closed ranges, where $H_3,H_4$ are Hilbert spaces. We shall furthermore assume the exact complex conditions
\[
   \rge(A_1)=\kar(A_2)\quad\text{and}\quad \rge(A_2)=\kar(A_3).
\]
\begin{example}\label{ex:2} Let $\Omega\subseteq \R^3$ be open and bounded. Assume that $\Omega$ is a simply connected weak Lipschitz domain with connected complement. We set
\begin{align*}
   H_0 &= L^2(\Omega) & H_1 & =L^2(\Omega)^3 \\ H_2 & =L^2(\Omega)^3 & H_3 & = L^2(\Omega) & H_4 & = \{0\} \\
   A_0 & =\grad, \dom(A_0) = H^1(\Omega) & A_1 & = \curl, \dom(A_1)= H(\curl) \\ A_2 & = \dive, \dom(A_2)=H(\dive) & A_3 & = 0, \dom(A_3)=H_3. & &
\end{align*}

\end{example}

\begin{theorem}\label{thm:st_abst_curldiv} Let $a^{-1}\in \mathcal{M}(\alpha,\beta,(A_0,A_1))$ and $b\in \mathcal{M}(\alpha,\beta,(A_{2},A_3))$.  For all $f\in \dom(\mathcal{A}_1^*)^*$ and $g\in \dom(\mathcal{A}_2)^*$ there exists a unique $u\in \dom({A}_1^*)\cap \dom(A_2)$ such that
\[
    \langle a A_1^*u,A_1^*\phi \rangle + \langle b A_2 u, A_2\psi\rangle = f(\phi)+g(\psi) \quad(\phi\in \dom(\mathcal{A}_1^*), \psi\in \dom(\mathcal{A}_2)).
\]
Moreoever, we have the continuity estimate
\[
   \|u\|_{H_2}+\|A_1^*u\|+\|A_2 u\|\leq \big(\frac1\alpha+\beta\big) c(\|f\|+\|g\|),
\]
where $c>0$ only depends on $A_1^*$ and $A_2$.
\end{theorem}
\begin{proof}
  We show uniqueness first. For this, let $f=0$ and $g=0$. We need to show that then necessarily $u=0$. By the assumptions on $a$ and $b$, we deduce that
  \[
       \pi_{1} A_1^*u=0,\quad \pi_2 A_2 u =0,
  \]
  where $\pi_1 \in \mathcal{B}(H_1)$ and $\pi_2\in \mathcal{B}(H_3)$ are the orthogonal projections onto $\rge(A_1^*)$ and $\rge(A_2)$, respectively. \red{For this, we detail the argument on how to get $\pi_{1} A_1^*u=0$ (the rationale is similar for $\pi_2 A_2 u =0$). Since $f=0$ and $g=0$, putting $\psi=0$ in the variational formulation, we infer
  \[
     \langle a A_1^*u,A_1^*\phi \rangle = 0 \quad(\phi\in \dom(\mathcal{A}_1^*)).
  \] Since $\rge(A_1^*|_{\dom(\mathcal{A}_1^*)})=\rge(A_1^*)$ (by definition), we deduce that 
  \[
      aA_1^*u \bot \rge(A_1^*)\text{,i.e., }\pi_1 aA_1^*u=0.
  \]
Using the block matrix representation of $a=\begin{pmatrix}a_{00} & a_{01} \\ a_{10} & a_{11}\end{pmatrix}$ according to the decompostion $\rge(A_0)\oplus\rge(A_1^*)$, we infer from $a^{-1}\in  \mathcal{M}(\alpha,\beta,(A_0,A_1))$ that $((a^{-1})^{-1})_{11}=a_{11}$ is an isomorphism on $\rge(A_1^*)$. Hence, $0=\pi_1 aA_1^*u = a_{11}\pi_1{A}_1^*u$ yields $\pi_1{A}_1^*u=0$.}

  Thus,
  \[
     A_1^* u=0 \quad A_2 u=0.
  \]
  Thus, $u\in \kar(A_1^*)\cap \kar(A_2)$. From $\rge(A_1)=\kar(A_2)$ we deduce $\kar(A_1^*)=\rge(A_2^*)$. Hence, $H_2= \rge(A_1)\oplus \kar(A_1^*)=\rge(A_1)\oplus \rge(A_2^*)$ and so $u\bot H_2$, which yields $u=0$. 
  
  For the existence part, we use again $H_2 = \rge(A_1) \oplus \rge(A_2^*)$. Let $f\in \dom(\mathcal{A}_1^*)^*$ and $g\in \dom(\mathcal{A}_2)^*$. By \cite[Theorem 3.1]{TW14} or \cite[Theorem 2.9]{W18}, we find $u_1 \in \dom(\mathcal{A}_1^*)=\dom(A_1^*)\cap \kar(A_1^*)^\bot=\dom(A_1^*)\cap \rge(A_1)$ with
 \[
     \langle a A_1^*u_1,A_1^*\phi \rangle = f(\phi)\quad (\phi\in \dom(\mathcal{A}_1^*))
 \]
 Similarly, we find $u_2\in \dom(\mathcal{A}_2)$ such that
 \[
     \langle b A_2u_2,A_2\psi \rangle = g(\psi)\quad (\psi\in \dom(\mathcal{A}_2)).
 \]
 Since $\rge(A_1)= \kar(A_2)$ and  $\rge(A_2^*)=\kar(A_1^*)$, we deduce that $u\coloneqq u_1 + u_2$ solves the problem in question.
 
 For the continuity estimate, we use the solution $u=u_1+u_2$ just constructed, to obtain
 \[
     |f(u_1)| =  |  \langle a A_1^*u_1,A_1^*u_1 \rangle |\geq \Re \langle a A_1^*u_1,A_1^*u_1\rangle \geq \frac1\beta \langle  A_1^*u_1,A_1^*u_1 \rangle \geq \frac1\beta(\frac12 \| A_1^*u_1\|^2+ \frac12 c_1\|u_1\|^2), 
 \]where $c_1 = \|(\mathcal{A}_1^*)^{-1}\|^{-1}$. A similar reasoning leads to an estimate of $|g(u_2)|$ from below. A combination of these yields the desired estimate.
\end{proof}
\begin{remark}
The proof of Theorem \ref{thm:st_abst_curldiv} does not really require the operators $A_0$ and $A_3$. The formulation of the theorem was done in the way above just for convenience with regards to the forthcoming applications.
\end{remark}

\section{The abstract convergence result\label{acgence}}

In this section, we shall use the operators $A_0,A_1,A_2,A_3$ as given as in the introductory part of the previous section. The main abstract result reads as follows.

\begin{theorem}\label{thm:conv_abs} Let $0<\alpha\leq \beta$, $(a_n^{-1})_n$, $a^{-1}$ in  $\mathcal{M}(\alpha,\beta,(A_0,A_1))$ and $(b_n)_n$, $b$  in $\mathcal{M}(\alpha,\beta,(A_2,A_3))$. Assume that $a_n^{-1}\to a^{{-1}}$ nonlocally $H$-converges w.r.t.~$(A_0,A_1)$ and that $b_n\to b$ nonlocally $H$-converges to $(A_2,A_3)$.  Given  $f\in \dom(\mathcal{A}_1^*)^*$ and $g\in \dom(\mathcal{A}_2)^*$ and let $(u_n)_n \in \dom(A_1^*)\cap\dom(A_2)$ satisfy
\[
   \langle a_n A_1^*u_n,A_1^*\phi \rangle + \langle b_n A_2 u_n, A_2\psi\rangle = f(\phi)+g(\psi) \quad(\phi\in \dom(\mathcal{A}_1^*), \psi\in \dom(\mathcal{A}_2)).
\]
Then $(u_n)_n$ weakly converges in $\dom(A_1^*)\cap\dom(A_2)$ to $u\in \dom(A_1^*)\cap\dom(A_2)$ satisfying
\[
  \langle a A_1^*u,A_1^*\phi \rangle + \langle b A_2 u, A_2\psi\rangle = f(\phi)+g(\psi) \quad(\phi\in \dom(\mathcal{A}_1^*), \psi\in \dom(\mathcal{A}_2)).
\]Moreoever, we have
\[
   a_n A_1^*u_n\rightharpoonup  aA_1^* u,\quad     b_n A_2u_n\rightharpoonup  b A_2 u
\]in $H_1$ and $H_3$, respectively.
\end{theorem}
\begin{proof}
  For $n\in \N$, we decompose $u_n=u_{n,1}+u_{n,2}$ for all $n\in \N$ with $u_{n,1}\in \dom(\mathcal{A}_1^*)$ and $u_{n,2}\in \dom(\mathcal{A}_2)$.  Putting $\psi=0$ in the variational equation satisfied by $u_n$, we obtain
  \[
   \langle a_n A_1^*u_{n,1},A_1^*\phi \rangle =  \langle a_n A_1^*u_n,A_1^*\phi \rangle =        f(\phi).
  \]
  As $a_n^{-1}$ nonlocally $H$-converges to $a^{-1}$, we obtain
  \[
       u_{n,1} \rightharpoonup u_1 \in \dom(\mathcal{A}_1^*)\text{ and } a_n A_1^*u_{n}=a_n A_1^*u_{n,1} \rightharpoonup a A_1^*u_1\in H_1,
  \]where $u_1$ is the unique solution of
  \[
     \langle a A_1^*u_{1},A_1^*\phi \rangle = f(\phi)\quad(\phi\in \dom(\mathcal{A}_1^*))
  \] Similarly, we deduce that $u_{n,2}$ weakly converges to some $u_2 \in \dom(\mathcal{A}_2)$, with $b_n A_2 u_n = b_n A_2 u_{n,2}\rightharpoonup bA_2 u_2$, where $u_2$ uniquely solves
  \[
        \langle b A_2u_{2},A_2\psi \rangle = g(\psi)\quad(\psi\in \dom(\mathcal{A}_2)).
  \]Thus, $v=u_1+u_2$ satisfies the same equation $u$ satisfies in the statement of the theorem. By Theorem \ref{thm:st_abst_curldiv}, we deduce $u=v$.
\end{proof}

The remainder of this section is devoted to show that the $H$-convergence conditions on $(a_n^{-1})_n$ and $(b_n)_n$ in the latter theorem are also necessary for the implied statement. For this we need another notion. 

\begin{definition} Let $\mathcal{A}\subseteq  \mathcal{B}(H_1)$, $A\colon \dom(A)\subseteq H \to H_1$, $H$ a Hilbert space.
We say that $A$ is \emph{w-identifying} for $\mathcal{A}$, if the following implication holds: Let $a_1,a_2 \in \mathcal{A}$ and assume that for all $u,\phi\in \dom(A)$ we have
\[
    \langle a_1A u,A \phi\rangle  =     \langle a_2 A u,A \phi\rangle.
\]Then $a_1 = a_2$.

We say that $A$ is \emph{s-identifying} for $\mathcal{A}$, if given $a_1,a_2 \in \mathcal{A}$ and $a_1A u=a_2 A u$ holds for all $u\in \dom(A )$, then $a_1=a_2$.
\end{definition}

\begin{remark} (a) The letters `w' and `s' are referring to `weakly' and `strongly', respectively. \red{However, s-identifying is weaker than w-identifying and, thus, w-identifying is stronger than s-identifying. Therefore, in order not to cause a wrong intuition, we choose to use `w-identifying' instead of 'weakly-identifying' and `s-identifying' instead of 'strongly-identifying'.}

(b) We shall see in the next section that the notion just introduced is particularly important for local operators; in the symmetric case `w-identifying' will be important and for non-symmetric operators `s-identifying' will be used.
\end{remark}

\begin{theorem}\label{thm:converse} Let $\mathcal{A}\subseteq \mathcal{M}(\alpha,\beta,(A_0,A_1))$, $\mathcal{B}\subseteq \mathcal{M}(\alpha,\beta,(A_2,A_3))$ be bounded subsets. Assume that $\mathcal{A}$ and $\mathcal{B}$ are closed under nonlocal $H$-convergence. Assume that $A_1^*$ is w-identifying for $\mathcal{A}^{-1}$ and that $A_2$ is w-identifying for $\mathcal{B}$. 

Let $(a_n^{-1})_n$, $a^{-1}$ in $\mathcal{A}$ and $(b_n)_n$, $b$ in $\mathcal{B}$.  Assume that for all $f\in \dom(\mathcal{A}_1^*)^*$ and $g\in \dom(\mathcal{A}_2)$ the following implication holds: Let $u_n\in \dom(A_1^*)\cap \dom(A_2)$ satisfy
\[
   \langle a_n A_1^* u_n,A_1^*\phi \rangle +    \langle b_n A_2 u_n,A_2\psi \rangle  = f(\phi)+g(\psi).
\]
Then $(u_n)_n$ weakly converges in $\dom(A_1^*)\cap\dom(A_2)$ to $u\in \dom(A_1^*)\cap\dom(A_2)$, where $u$ satisfies
\[
    \langle a A_1^* u,A_1^*\phi \rangle +    \langle b A_2 u,A_2\psi \rangle  = f(\phi)+g(\psi).
\]

Then $(a_n^{-1})_n$ and $(b_n)_n$ nonlocally $H$-converge to $a^{-1}$ and $b$, respectively.
\end{theorem}

For the proof of the latter statement, we need to recall a compactness result from \cite{W18}.

\begin{theorem}[{{\cite[Theorem 5.5 and Remark 5.6]{W18}}}]\label{thm:comp} Let $\mathcal{B}\subseteq \mathcal{M}(\alpha,\beta,(A_0,A_1))$ bounded. Then for every sequence in $\mathcal{B}$, we find a nonlocally $H$-convergent subsequence.
\end{theorem}

\begin{proof}[Proof of Theorem \ref{thm:converse}] By Theorem \ref{thm:comp}, we find a subsequence $(n_k)_k$ such that both $(a_{n_k}^{-1})_k$ and $(b_{n_k})_k$ nonlocally $H$-converge to some $\tilde{a}^{-1}$ and $\tilde{b}$. By closedness of $\mathcal{A}$ and $\mathcal{B}$, we obtain $\tilde{a}^{-1}\in \mathcal{A}$ and $\tilde{b}\in \mathcal{B}$. Let $v\in \dom(A_1^*)\cap\dom(A_2)$ and define $f(\phi)= \langle a A_1^* v,A_1^*\phi \rangle$ and $g(\psi)=   \langle b A_2 v,A_2\psi \rangle$. Let $u_n\in  \dom(A_1^*)\cap\dom(A_2)$ satisfy
\[
   \langle a_n A_1^* u_n,A_1^*\phi \rangle +    \langle b_n A_2 u_n,A_2\psi \rangle  = f(\phi)+g(\psi).
\]
By assumption, we deduce that $(u_n)_n$ weakly converges to some $u\in \dom(A_1^*)\cap\dom(A_2)$ satisfying
\[
      \langle a A_1^* u,A_1^*\phi \rangle +    \langle b A_2 u,A_2\psi \rangle  = f(\phi)+g(\psi) =  \langle a A_1^* v,A_1^*\phi \rangle +  \langle b A_2 v,A_2\psi \rangle.
\]
By Theorem \ref{thm:st_abst_curldiv}, we deduce that $u=v$. In consequence, we get that $u_{n_k}\rightharpoonup v$. Thus, by Theorem \ref{thm:conv_abs}, we obtain
\[
        \langle \tilde{a} A_1^* v,A_1^*\phi \rangle +    \langle \tilde{b} A_2 v,A_2\psi \rangle = \langle a A_1^* v,A_1^*\phi \rangle +  \langle b A_2 v,A_2\psi \rangle,
\]
which implies $\tilde{a}=a$ and $\tilde{b}=b$ since $A_1^*$ and $A_2$ are w-identifying. A subsequence principle, cf.~Theorem \ref{thm:NHC_top}, concludes the proof.
\end{proof}

The non-symmetric case reads as follows.

\begin{theorem}\label{thm:converse2} Let $\mathcal{A}\subseteq \mathcal{M}(\alpha,\beta,(A_0,A_1))$, $\mathcal{B}\subseteq \mathcal{M}(\alpha,\beta,(A_2,A_3))$ be bounded subsets. Assume that $\mathcal{A}$ and $\mathcal{B}$ are closed under nonlocal $H$-convergence. Assume that $A_1^*$ is s-identifying for $\mathcal{A}^{-1}$ and that $A_2$ is s-identifiying for $\mathcal{B}$. 

Let $(a_n^{-1})_n$, $a^{-1}$ in $\mathcal{A}$ and $(b_n)_n$, $b$ in $\mathcal{B}$.  Assume that for all $f\in \dom(\mathcal{A}_1^*)^*$ and $g\in \dom(\mathcal{A}_2)^*$ the following implication holds: Let $u_n\in \dom(A_1^*)\cap \dom(A_2)$ satisfy
\[
   \langle a_n A_1^* u_n,A_1^*\phi \rangle +    \langle b_n A_2 u_n,A_2\psi \rangle  = f(\phi)+g(\psi).
\]
Then $(u_n)_n$ weakly converges in $\dom(A_1^*)\cap\dom(A_2)$ to $u\in \dom(A_1^*)\cap\dom(A_2)$, $a_nA_1^*u_n \rightharpoonup aA_1^*u$ and $b_nA_2u_n \rightharpoonup bA_2u$ , where $u$ satisfies
\[
    \langle a A_1^* u,A_1^*\phi \rangle +    \langle b A_2 u,A_2\psi \rangle  = f(\phi)+g(\psi).
\]

Then $(a_n^{-1})_n$ and $(b_n)_n$ nonlocally $H$-converge to $a^{-1}$ and $b$, respectively.
\end{theorem}
\begin{proof}
The proof follows the same lines as the proof of Theorem \ref{thm:converse}. Indeed, with the same notation as in the proof of Theorem \ref{thm:converse}, the only fundamental difference is that one uses the full statement of Theorem \ref{thm:conv_abs} of the convergence of the `fluxes', i.e., $a_nA_1^*u_n \rightharpoonup aA_1^*u$ and $b_nA_2u_n \rightharpoonup bA_2u$, in order to deduce that $\tilde{a}A_1^*v=aA_1^*v$ and $\tilde{b}A_2v=bA_2v$; which eventually leads to $\tilde{a}=a$ and $\tilde{b}=b$. Again an application of Theorem \ref{thm:NHC_top} yields the assertion.
\end{proof}

\section{An application to curl-div-systems with local coefficients\label{s:applilocal}}

In this section, we elaborate on the implications of the above formulated abstract results in the particular setting of Example \ref{ex:2}. 

We describe the nonlocal $H$-convergence topologies in more detail in the next proposition.

\begin{proposition}\label{prop:NHC_spec} 
(a) Let $(b_n)_n$ and $b$ in $\mathcal{M}(\alpha,\beta,(\dive,0))$. Then $(b_n)_n\to b$ $H$-nonlocally, if and only if $b_n^{-1}\to b^{-1}$ in the weak operator topology of $\mathcal{B}(L^2(\Omega))$.

(b) Let $(a_n)_n$ in $L^\infty(\Omega)^{3\times 3}$ such that $\Re a_n\geq \frac1\beta$ and $\Re a_n^{-1}\geq \alpha$ and, $a$ similarly. Then $(a_n^{-1})_n\to a^{-1}$ $H$-nonlocally \red{(w.r.t~$(\grad,\curl)$)}, if and only if $(a_n^{-1})_n$ (locally) $H$-converges to $a$, that is, for all $f\in H^{-1}(\Omega)$ and $u_n \in H_0^1(\Omega)$ satisfying
\[
    -\dive a_n^{-1} \grad u_n = f,
\]we have $u_n\rightharpoonup u \in H_0^1(\Omega)$ and $a_n^{-1}\grad u_n \rightharpoonup a^{-1}\grad u\in L^2(\Omega)^3$, where $u$ satisfies
\[
      -\dive a^{-1} \grad u = f.
\]
(c) Let $(a_n)_n=(a_n^*)_n$ in $L^\infty(\Omega)^{3\times 3}$ such that $ a_n\geq \frac1\beta$ and $ a_n^{-1}\geq \alpha$ and, $a$ similarly. Then $(a_n^{-1})_n\to a^{-1}$ $H$-nonlocally \red{(w.r.t.~$(\grad,\curl)$)}, if and only if $(a_n^{-1})_n$ (locally) $G$-converges to $a$, that is, for all $f\in H^{-1}(\Omega)$ and $u_n \in H_0^1(\Omega)$ satisfying
\[
    -\dive a_n^{-1} \grad u_n = f,
\]we have $u_n\rightharpoonup u \in H_0^1(\Omega)$, where $u$ satisfies
\[
      -\dive a^{-1} \grad u = f.
\]
\end{proposition}
\begin{proof}
 For (a), we use that $\dive$ is surjective onto $L^2(\Omega)$. Thus, the assertion follows upon relying onto Theorem \ref{thm:NHC_top} (note that $\rge(\pi^*)=\{0\}$ is trivial in this case).
 
 The assertions (b) and (c) have been shown in \cite[Remark 4.11, Theorem 5.11 and Remark 5.12]{W18}.
\end{proof}
 \red{
 \begin{remark}\label{rem:anEn}
 In the situation of (b), assume that $(a_n^{-1})_n$ locally $H$-converges to some $a^{-1}$. By the equivalence asserted in (b), it then follows that $(a_n^{-1})_n$ nonlocally $H$-converges to $a^{-1}$ w.r.t.~$(\grad,\curl)$. In particular, this means that given $f\in \dom(\mathcal{A}_1^*)^*=\mathcal{H}_0(\curl)^*$ (see \eqref{eq:curl0} below) and $v_n\in \mathcal{H}_0(\curl)$ solving
 \[
       \langle a_n\curl v_n,\curl \phi\rangle = f(\phi)\quad(\phi\in \mathcal{H}_0(\curl)),
 \] we obtain that $v_n\rightharpoonup v$ weakly in $\mathcal{H}_0(\curl)$ and $ a_n\curl v_n\rightharpoonup a\curl v$ in $L^2(\Omega)^3$, where $v\in\mathcal{H}_0(\curl)$ is the unique solution of
  \[
       \langle a\curl v,\curl \phi\rangle = f(\phi)\quad(\phi\in \mathcal{H}_0(\curl)).
 \]
 \end{remark}
 }
 \begin{proposition}\label{prop:ws_ident} (a) $\dive\colon H(\dive)\subseteq L^2(\Omega)^3\to L^2(\Omega)$ is 
 {s-identifying and} w-identifying for $\mathcal{M}(\alpha,\beta,(\dive, {0}))$.
 
 (b) $\curl\colon H_0(\curl)\subseteq L^2(\Omega)^3\to L^2(\Omega)^3$ is w-identifying for 
{${M}_{\textnormal{sym}}(\alpha,\beta,\Omega)^{-1}$ with
\[{M}_{\textnormal{sym}}(\alpha,\beta,\Omega):=\{ a\in L^\infty(\Omega;\R)^{3\times 3}; a=a^T,  a\geq \alpha, a^{-1}\geq 1/\beta\}.\]}
 
 (c) $\curl\colon H_0(\curl)\subseteq L^2(\Omega)^3\to L^2(\Omega)^3$ is s-identifying for 
{${ M}(\alpha,\beta,\Omega)^{-1}$ with
 \[{M}(\alpha,\beta,\Omega):=\{ a\in L^\infty(\Omega; \C)^{3\times 3}; \Re a \geq \alpha, \Re a^{-1}\geq 1/\beta\}. \]}
 \end{proposition}
 \begin{proof}
 (a) follows from the fact that $\dive$ maps onto $L^2(\Omega)$.
 
 For the assertion in  (c), let $a_1^{-1}, a_2^{-1}$ be in $\{ a^{-1}\in L^\infty(\Omega; \C)^{3\times 3};\Re a\geq 1/\beta, \Re a^{-1}\geq \alpha\}$ and such that
 \begin{equation}\label{assumpc}
\forall u\in H_0(\curl)\colon a_1^{-1} \curl  u= a_2^{-1} \curl  u.
 \end{equation}
 We recall that for a scalar field $p$ and a vectorial one $A$ (smooth enough), one has
 \begin{equation}\label{curlproduct}
 \curl (p A)=p\curl A+(\nabla p)\times A.
 \end{equation}
 For a fixed $x_0\in \Omega$ and $\varepsilon>0$ small enough such that  $\overline{B(x_0,\varepsilon)}\subset \Omega$, we consider $\tau_{0}\in C^\infty_c(\Omega)$ such that $\tau_{0}$ is equal to 1 in 
 the ball $B(x_0,\varepsilon)$.
 Applying the property \eqref{curlproduct} with $p=\tau_{0}$ and the function $A_b$ defined by
 \[
 A_b(x)= b\times x\quad(x\in \mathbb{R}^3),
 \]
 for an arbitrary $b\in \mathbb{R}^3$,
 we have
 \[
 \curl(\tau_0 A_b)=2\tau_0 b+(\nabla \tau_0)\times A_b \hbox{ in } \Omega,
 \]
 since
 $\curl (b\times x)=2b$.
 Since $\tau_0 A_b$ belongs to $H_0(\curl)$, by our assumption  \eqref{assumpc}, we then have
 \[
 a_1^{-1}(2\tau_0 b+(\nabla \tau_0)\times A_b)=a_2^{-1}(2\tau_0 b+(\nabla \tau_0)\times A_b).
 \]
 Restricting this identity to $B(x_0,\varepsilon)$, we obtain
  \[
 a_1^{-1}  b =a_2^{-1} b \hbox{ on } B(x_0,\varepsilon).
 \]
 Since this identity holds for all $b\in \mathbb{R}^3$, we deduce that
  \[
 a_1^{-1}  =a_2^{-1} \hbox{ on } B(x_0,\varepsilon).
 \]
 This proves that $a_1^{-1}  =a_2^{-1}$ because $x_0$ is arbitrary in $\Omega$.

 For the assertion in (b), let  $a_1^{-1}, a_2^{-1}$ be in $\{a^{-1}\in L^\infty(\Omega;\R)^{3\times 3}; a=a^T, a\geq 1/\beta, a^{-1}\geq \alpha\}$ and such that
 \begin{equation}\label{assumpb}
 \langle a_1^{-1}\curl u,\curl v\rangle_{L^2(\Omega)^3}= \langle a_2^{-1}\curl u, \curl v\rangle_{L^2(\Omega)^3}  \quad( u, v\in H_0(\curl)).
 \end{equation}
 We let $\tau\in C^\infty_c(\Omega)$, $\xi\in \mathbb{R}^3$, $b\in \mathbb{C}^3$ and,  similarly to \cite[Example 6.7]{TGW18}, for all real number $\lambda\geq 1$, we define
 $u_{\tau,b,\xi,\lambda}\in C^\infty_c(\Omega)^3$ by
 \[
 u_{\tau,b,\xi,\lambda}(x)=\tau(x) e^{i\lambda \xi\cdot x}b, \forall x\in \Omega.
 \]
  Then again using \eqref{curlproduct} we have
 \[
 \curl  u_{\tau,b,\xi,\lambda}= i\lambda \tau   e^{i\lambda \xi\cdot x}(\xi\times b)
 + e^{i\lambda \xi\cdot x} (\nabla \tau)\times b\hbox{ in } \Omega.
 \]
 Then, setting $d=a_1^{-1}-a_2^{-1}$,  by  \eqref{assumpb}, we have
 \[
 \langle d   \curl  u_{\tau,b,\xi,\lambda},  \curl  u_{\tau',b,\xi,\lambda}\rangle_{L^2(\Omega)^3}=0,
 \]
 for all $\tau,\tau'\in C^\infty_c(\Omega)$, $\xi\in \mathbb{R}^3$, $b\in \mathbb{C}^3$,  and $\lambda\geq 1$.
 Hence by the previous identity, we find that
  \[
 \langle d( i\lambda \tau    (\xi\times b)
 +  (\nabla \tau)\times b),  i\lambda \tau'  ( \xi\times b)
 + (\nabla \tau')\times b\rangle_{L^2(\Omega)^3}=0.
 \]
 Dividing by $\lambda^2$ and letting $\lambda$ go to infinity, we find
 \[
 \int_\Omega \tau(x) \tau'(x) \langle d(x)  (\xi\times b), \xi\times b\rangle_{\mathbb{C}^3}\,dx=0,
 \]
  for all $\tau,\tau'\in C^\infty_c(\Omega)$, $\xi\in \mathbb{R}^3$, and  $b\in \mathbb{C}^3$ . Since $C^\infty_c(\Omega)$ is dense in $L^2(\Omega)$, we deduce that
 \[
 \langle d(x)  (\xi\times b), \xi\times b\rangle_{\mathbb{C}^3}=0 \hbox{ for almost all } x\in \Omega,
 \]
and all $\xi\in \mathbb{R}^3$ and $b\in \mathbb{C}^3$ . As  $\{\xi\times b: \xi\in \mathbb{R}^3, b\in \mathbb{C}^3\}=\mathbb{C}^3$
 and
 since $d$ is selfadjoint, we conclude that $d=0$.
 \end{proof}

\begin{remark}\label{rem:ident}
  A closer look at the proofs of the statements in Proposition \ref{prop:ws_ident}, we see that the positive definiteness conditions were not used. In fact, it even suffices to restrict the $\curl$ operators to $C_c^\infty(\Omega)$-vector fields and the $\dive$-operator to $C^\infty(\overline{\Omega})$-vector fields, whenever they are dense in $H(\dive)$, which is for instance the case if $\Omega$ has a continuous boundary.
\end{remark}

We conclude this section with a closer look at periodic material coefficients.

\begin{theorem}[{{\cite[Theorem 2.6, Theorem 6.1, Theorem 5.12]{CioDon}}}]\label{thm:hompre}
(a) Let $b\in L^\infty(\R^3)$ satisfy $b(x+k)=b(x)$ for all $k\in \Z^3$ and almost all $x\in \R^3$. Then $b_n\coloneqq b(n\cdot)\to \m(b)\coloneqq \int_{[0,1)^3}b(x)dx$ in $\sigma(L^\infty(\Omega),L^1(\Omega))$ for all $\Omega\subseteq \R^3$ open \footnote{We recall that $\sigma(L^\infty(\Omega),L^1(\Omega))$ denotes the weak*-topology on $L^\infty(\Omega)$; if $\Omega$ is clear from the context, we shall also use the shorthand $\sigma(L^\infty,L^1)$.}.

 (b) Let $a\in L^\infty(\R^3)^{3\times 3}$ satisfying $\Re a(x)\geq \gamma>0$ and $a(x+k)=a(x)$ for all $k\in \Z^3$ and almost all $x\in \R^3$. Define $a_n(x)\coloneqq a(n x)$, $n\in \N$, $x\in \R^3$.
Then $(a_n)_n$ $H$-converges to some $a_{\textnormal{hom}}\in \C^{3\times 3}$.

(c) Let $\tilde{a} \in L^\infty(\R)^{3\times 3}$ and assume that $a\colon (x_1,x_2,x_3)\mapsto \tilde{a}(x_1)$ satisfies the conditions in (b). Then 
\begin{multline*}
     a_{\textnormal{hom}}=\frac{1}{\m\left(\frac{1}{a_{11}}\right)} \begin{pmatrix}1 &  \m\big(\frac {a_{12}}{a_{11}}\big) & \m\big(\frac {a_{13}}{a_{11}}\big) \\
                                                                  \m\big(\frac {a_{21}}{a_{11}}\big) &  \m\big(\frac {a_{21}}{a_{11}}\big)  \m\big(\frac {a_{12}}{a_{11}}\big)     &     \m\big(\frac {a_{21}}{a_{11}}\big) \m\big(\frac {a_{13}}{a_{11}}\big)   
                                                                                                                                          \\
                                                                 \m\big(\frac {a_{31}}{a_{11}}\big)&        \m\big(\frac {a_{31}}{a_{11}}\big) \m\big(\frac {a_{12}}{a_{11}}\big)                         &         \m\big(\frac {a_{31}}{a_{11}}\big) \m\big(\frac {a_{13}}{a_{11}}\big)                                 \end{pmatrix}\\ + \begin{pmatrix} 0 & 0 & 0 \\ 0 & \m\big(a_{22} - \frac { a_{21}a_{12}}{a_{11}}\big) &  \m\big(a_{23} - \frac { a_{21}a_{13}}{a_{11}}\big) \\ 0 &  \m\big(a_{32} - \frac { a_{31}a_{12}}{a_{11}}\big)   &   \m\big(a_{33} - \frac { a_{31}a_{13}}{a_{11}}\big)                        \end{pmatrix}.
\end{multline*} 
\end{theorem}

For the next theorem, we introduce the spaces
\begin{align}\label{eq:curl0}
\mathcal{H}_0(\curl) & \coloneqq  \{\phi\in H_0(\curl); \phi \in \curl[H(\curl)]\} \\
\mathcal{H} (\dive) & \coloneqq \{\psi\in H(\dive); \psi \in \grad[H_0^1(\Omega)]\}
\end{align}
considered as subspaces of $H(\curl)$ and $H(\dive)$. Note that these spaces are Hilbert spaces in the following as both $ \curl[H(\curl)]$ and $ \grad[H_0^1(\Omega)]$ are closed due to the assumptions on $\Omega$.

\begin{theorem}\label{thm:homloc} Let $\tilde{d}\in L^\infty(\R)$ with $d\colon (x_1,x_2,x_3)\mapsto \tilde{d}(x_1) $ satisfying the assumptions on $a$ in Theorem \ref{thm:hompre}(b). Put $d_n\coloneqq d(n\cdot)$. Let $\Omega\subseteq \R^3$ open, bounded, simply connected, weak Lipschitz domain with connected complement. For $n\in \N$ consider the problem of finding $u_n\in H_0(\curl)\cap H(\dive)$ such that for given $f\in \mathcal{H}_0(\curl)^*$ and $g\in  \mathcal{H}(\dive)^*$ we have
\begin{equation*}
     \langle d_n \curl u_n, \curl \phi\rangle + \langle d_n \dive u_n,\dive \psi \rangle  = f(\phi)+g(\psi)\quad \big(\phi\in  \mathcal{H}_0(\curl), \ \psi\in \mathcal{H}(\dive)\big).
\end{equation*}
Then $u_n \rightharpoonup u\in H_0(\curl)\cap H(\dive)$, where $u$ satisfies
\begin{multline*}
      \langle \begin{pmatrix} {\m(d)} & 0 & 0 \\ 0 & \frac1{\m\big(\frac1d\big)} & 0  \\ 0 & 0   &  \frac1{ \m\big(\frac1d\big)} \end{pmatrix}\curl u, \curl \phi\rangle + \langle \frac1{\m\big(\frac1d\big)} \dive u,\dive \psi \rangle \\ = f(\phi)+g(\psi)\quad \big(\phi\in  \mathcal{H}_0(\curl), \ \psi\in \mathcal{H}(\dive)\big).
\end{multline*}Moreover, we have 
\begin{align*}
d_n\curl u_n & \rightharpoonup  \begin{pmatrix} \m(d) & 0 & 0 \\ 0 &  \frac{1}{\m\big(\frac1d\big) }&  0 \\ 0 & 0   &    \frac{1}{\m\big(\frac1d\big) }\end{pmatrix}\curl u \\
d_n \dive u_n & \rightharpoonup  \frac{1}{ \m\big(\frac1d\big) }\dive u.
\end{align*}
\end{theorem}
\begin{proof}
 \red{We apply the abstract convergence result Theorem \ref{thm:conv_abs} to the situation mentioned in Example \ref{ex:2}. Then, $A_0=\grad$ and $A_1=\curl$. In particular, we then have that $A_1^*=\curl$ with $\dom(A_1^*)=H_0(\curl)$. As a consequence, we infer $\mathcal{A}_1^*=\curl$ with $\dom(\mathcal{A}_1^*)=\mathcal{H}_0(\curl)$ and $A_2=\dive$ with $\dom(A_2)=H(\dive)$ leads to $\dom(\mathcal{A}_2)=\mathcal{H}(\dive)$. In the situation of Theorem \ref{thm:conv_abs}, we let
 \[
     a_n \coloneqq \begin{pmatrix} d_n & 0 & 0 \\ 0 & d_n & 0 \\ 0 & 0 & d_n \end{pmatrix}\text{ and }b_n \coloneqq d_n\quad(n\in \mathbb{N}).
 \]
 In order to conclude the proof it remains to check whether
 \begin{equation}\label{eq:aninv}
     a_n^{-1} =\begin{pmatrix} d_n^{-1} & 0 & 0 \\ 0 & d_n^{-1} & 0 \\ 0 & 0 & d_n^{-1} \end{pmatrix} \to \begin{pmatrix} {\m(d)}^{-1} & 0 & 0 \\ 0 & \left(\frac1{\m\big(\frac1d\big)}\right)^{-1} & 0  \\ 0 & 0   &  \left(\frac1{ \m\big(\frac1d\big)}\right)^{-1} \end{pmatrix}
 \end{equation} $H$-nonlocally with respect to $(\grad,\curl)$ and
 \begin{equation}\label{eq:dn}
     b_n=d_n \to  \frac{1}{ \m\big(\frac1d\big) }
 \end{equation} $H$-nonlocally with respect to $(\dive,0)$. Since $1/d_n\to \m\big(\frac1d\big)$ as $n\to\infty$ due to the periodicity of $d$ in the weak operator topology (see Theorem \ref{thm:hompre} (a)), we obtain the convergence in \eqref{eq:dn} by Proposition \ref{prop:NHC_spec} (a). 
\\
Next, for the proof of \eqref{eq:aninv}, we note that $(a_n^{-1})_n$ locally $H$-converges to some $\tilde a$ by  Theorem \ref{thm:hompre} (b). From  Theorem \ref{thm:hompre} (c), we deduce that
\[
   \tilde a=\begin{pmatrix} {\m(d)}^{-1} & 0 & 0 \\ 0 & \left(\frac1{\m\big(\frac1d\big)}\right)^{-1} & 0  \\ 0 & 0   &  \left(\frac1{ \m\big(\frac1d\big)}\right)^{-1} \end{pmatrix}.\] Hence, we obtain \eqref{eq:aninv} upon using Proposition \ref{prop:NHC_spec} (b).
 }
\end{proof}

\begin{theorem}\label{thm:conversecurldiv} 
Let $(a_n^{-1})_n$, $a^{-1}$ in $M_{\textnormal{sym}}(\alpha,\beta,\Omega)$  and $(b_n)_n$, $b$ in $\mathcal{M}(\alpha,\beta,(\dive,0))$.  Assume that for all $f\in H_0(\curl)^*$ and $g\in H(\dive)^*$ the following implication holds: Let $u_n\in H_0(\curl)\cap H(\dive)$ satisfy
\[
   \langle a_n \curl u_n, \curl\phi \rangle +    \langle b_n \dive u_n, \dive\psi \rangle  = f(\phi)+g(\psi) \quad\big(\phi\in  \mathcal{H}_0(\curl), \ \psi\in \mathcal{H}(\dive)\big).
\]
Then $(u_n)_n$ weakly converges in $v$ to $u\in H_0(\curl)\cap H(\dive)$, where $u$ satisfies
\[
    \langle a\curl  u,\curl \phi \rangle +    \langle b \dive u, \dive\psi \rangle  = f(\phi)+g(\psi) \quad \big(\phi\in  \mathcal{H}_0(\curl), \ \psi\in \mathcal{H}(\dive)\big).
\]
Then $(a_n^{-1})_n$ and $(b_n)_n$ nonlocally $H$-converge to $a^{-1}$ and $b$, respectively.
\end{theorem}
\begin{proof}
Direct consequence of Theorems \ref{thm:converse} and \ref{thm:converse2}  using the statements (a) and (b) (resp. (a) and  (c)) of Proposition \ref{prop:ws_ident}.
\end{proof}

\section{An application to curl-div-systems with nonlocal coefficients\label{s:applinonlocal}}

In this section, we elaborate on the implications of the above formulated abstract results in the particular setting of Example \ref{ex:2} for nonlocal coefficients. For this, as in the previous section let $\Omega\subseteq \R^3$ be open, bounded, simply connected, weak Lipschitz with connected complement. We consider $k \in L^\infty(\R^3)$ $[0,1)^3$-periodic with $\|k\|_{L^\infty}<\lambda(\Omega)$ and define
\[
     k_n*f(x) \coloneqq \int_{\Omega} k(n(x-y))f(y) dy
\] for all $f\in L^2(\Omega)\cup L^2(\Omega)^3$.  Young's inequality confirms that $\sup_n\|k_n*\|_{\mathcal{B}(L^2(\Omega))}<1$.

Using the techniques from \cite[Example 6.7]{W18}, see also \cite[Example 3.9]{W19}, we arrive at the following result.

\begin{theorem}[{{\cite[Example 3.9]{W19}}}] Let $f\in \mathcal{H}_0(\curl)^*$ and $g\in \mathcal{H}(\dive)^*$. For $n\in \N$ let $u_n\in H_0(\curl)\cap H(\dive)$ satisfy
\begin{multline*}
     \langle (1-k_n*)^{-1}\curl u_n, \curl \phi\rangle + \langle (1-k_n*)^{-1} \dive u_n,\dive \psi \rangle \\ = f(\phi)+g(\psi)\quad\big(\phi\in \mathcal{H}_0(\curl), \ \psi\in \mathcal{H}(\dive)\big).
\end{multline*}
Then $u_n \rightharpoonup u\in H_0(\curl)\cap H(\dive)$, where $u$ is the unique solution of
\begin{multline*}
        \langle (1-\tilde{k}*)^{-1}\curl u, \curl \phi\rangle + \langle (1-\tilde{k}*)^{-1} \dive u,\dive \psi \rangle \\ = f(\phi)+g(\psi)\quad\big(\phi\in \mathcal{H}_0(\curl), \ \psi\in \mathcal{H}(\dive)\big).
\end{multline*}
Moreover, we find the convergence of the fluxes 
\begin{align*}
 (1-k_n*)^{-1}\curl u_n & \rightharpoonup   (1-\tilde{k}*)^{-1}\curl u \\
 (1-k_n*)^{-1} \dive u_n & \rightharpoonup   (1-\tilde{k}*)^{-1}\dive u,
\end{align*}
where $\tilde{k}=\m(k)\chi_{\Omega}$.
\end{theorem}

It is interesting to see that the behaviour of the continuous dependence is different from the local coefficient case. The reason for this is that the convolution with $k_n$ adds a hidden compactness to the problem.

\section{Impedance boundary conditions and applications to scattering\label{s:impbc}}

In the following we want to address convergence results also for $\curl$-$\dive$-problems with impedance type boundary conditions. For this, however, the considered system needs to be amended and extended. Thus, before we actually come to the convergence result -- similar to the consideration of the previous problems -- we shall focus on well-posedness conditions first.

\subsection{Functional analytic preliminaries}

Throughout this section let $\Omega\subseteq \R^3$ be an open bounded set with Lipschitz boundary $\Gamma\coloneqq \partial\Omega$. We recall the following fact from the literature  (see for instance \cite{BuffaCostabelSchwab:02,BuffaCostabelSheen}).
Introduce the mappings
\begin{eqnarray*}
\gamma_\tau: H^1(\Omega)^3\to \mathbb{L}^2_t(\Gamma), E\to E\times n,\\
\pi_\tau: H^1(\Omega)^3\to \mathbb{L}^2_t(\Gamma), E\to n\times(E\times n),\\
\end{eqnarray*}
where $\mathbb{L}^2_t(\Gamma)$ is defined by
\[
\mathbb{L}^2_t(\Gamma)=\{E\in (L^2(\Gamma))^3: E\cdot n=0 \hbox{ a.e. }\}.
\]
We further denote by
$V_\gamma=\gamma_\tau (H^1(\Omega)^3)$ and $V_\pi=\pi_\tau (H^1(\Omega)^3)$, the range of these mappings
that are Hilbert spaces with respective norms:
\begin{eqnarray*}
\|\lambda\|_{V_\gamma}=\inf_{E\in H^1(\Omega)^3}\{\|E\|_{H^1(\Omega)^3}: \gamma_\tau E=\lambda\},\\
\|\lambda\|_{V_\pi}=\inf_{E\in H^1(\Omega)^3}\{\|E\|_{H^1(\Omega)^3}: \pi_\tau E=\lambda\}.
\end{eqnarray*}

\begin{theorem}[{\cite[p. 855]{BuffaCostabelSchwab:02}}]\label{thm:trace} The mapping  $\gamma_{\tau}$
(resp. $\pi_\tau$) extends continuously to the whole of $H(\curl)$ into $V_\pi'$ (resp. $V_\gamma'$).
\end{theorem}

In the following, we will denote the continuous extensions of $ \gamma_{\tau}$ and $ \pi_{\tau}$ by their same name.

We denote by $A_0\colon \dom(A_0)\subseteq L^2(\Omega)^6 \to L^2(\Omega)^8$ the operator acting like
\[
   A_0 (E,H) \coloneqq  \begin{pmatrix} \curl E \\ \dive E \\ \curl H \\ \dive H   
   \end{pmatrix},
\]
for $(E,H)$ in
\[
   \dom(A_0) \coloneqq \{ (E,H)\in H(\curl)\cap H(\dive);  \gamma_{\tau} H =  \pi_{\tau} E\}.
\]
\red{Note that the equality  $\gamma_{\tau} H =  \pi_{\tau} E$ in the definition of the domain of $\dom(A_0)$ comprises an implicit regularity statement. In order for $\gamma_{\tau} H$ and $  \pi_{\tau} E$ to be equal they need to belong to the same space, which is in this case $\mathbb{L}^2_t(\Gamma)$.}
\begin{theorem} The operator $A_0$ is densely defined and closed.
\end{theorem}
\begin{proof} Since $C_c^\infty(\Omega)^6\subseteq \dom(A_0)$ the operator $A_0$ is densely defined. For the closedness of $A_0$, let $(E_n,H_n)_n$ be a convergent sequence in $L^2(\Omega)^6$ so that $(A_0(E_n,H_n))_n$ converges on $L^2(\Omega)^8$. Denote the respective limits by $(E,H)$ and $(E_c,E_d,H_c,H_d)$. By the closedness of $\curl$ and $\dive$, we obtain $(E_c,E_d,H_c,H_d)=( \curl E ,\dive E ,\curl H ,\dive H  )$. In particular, we obtain that both $(E_n)_n$ and $(H_n)_n$ converge to $E$ and $H$ in $H(\curl)$. Thus, by Theorem \ref{thm:trace}, we may let $n\to\infty$ in the equality
\[
    \gamma_{\tau} H_n =     \pi_{\tau} E_n
\]
and obtain
\[
  \gamma_{\tau} H =    \pi_{\tau} E,
\]which eventually shows that $(E,H)\in \dom(A_0)$. The claim follows.
\end{proof}

We need the following notion.

\begin{definition} We say that $\Omega$ has the \emph{impedance compactness property}, if $\dom(A_0)\hookrightarrow L^2(\Omega)^6$ compactly.
\end{definition}

\begin{theorem}[\cite{GLCI,SNJT_polyhedral}]\label{thm:embH1} Let $\Omega$ be $C^2$-domain or a convex polyhedron. Then $\dom(A_0)\hookrightarrow H^1(\Omega)^6$.
\end{theorem}

\begin{corollary} Let $\Omega$ be $C^2$ or a convex polyhedron. Then $\Omega$ has the impedance compactness property.
\end{corollary}
\begin{proof}
The conditions imply that $\Omega$ has continuous boundary. Thus, $H^1(\Omega)\hookrightarrow L^2(\Omega)$ compactly, by the Rellich--Kondrachov selection theorem. Hence, the claim follows from Theorem \ref{thm:embH1}.
\end{proof}

\begin{remark} The reason we have introduced the notion of `impedance compactness property' is due to the fact that in all likelihood  the $H^1$-detour to show the compact embedding result is not needed. In fact, detouring $H^1$ has led to compact embedding results for $H_0(\curl)\cap H(\dive)$ and $H(\curl)\cap H_0(\dive)$ (and even mixed boundary conditions) for $\Omega$ being only \emph{weak} Lipschitz domains, that is, Lipschitz manifolds (`Picard--Weck selection theorems'), see \cite{W74,W80,P84} also see \cite{BPS16} for different boundary conditions.
\end{remark}

\begin{lemma}\label{lem:clr} Let $\Omega$ have the impedance compactness property. Then $\kar(A_0)$ is finite-dimensional and $\rge(A_0)\subseteq L^2(\Omega)^8$ is closed. Moreover, there exists $c>0$ such that for all $\phi\in \dom(A_0)\cap \kar(A_0)^\bot$
\[
     c\|\phi\|\leq \|A_0\phi\|.
\]
\end{lemma}
\begin{proof}
The property that $\dom(A_0)$ embeds compactly into $L^2(\Omega)^6$ implies that the unit ball of $\kar(A_0)$ embeds compactly into $L^2(\Omega)^6$. Since the norms on $\kar(A_0)$ and $L^2(\Omega)^6$ coincide, $\kar(A_0)$ is necessarily finite-dimensional. The closedness of the range as well as the asserted inequality are standard consequences of the compactness of $\dom(A_0)\hookrightarrow L^2(\Omega)^6$; see \cite[Lemma 4.1]{TGW18} applied to $G=A_0$.
\end{proof}

\begin{remark} In the case of $\Omega$ being a convex polyhedron, it has been shown in \cite{SNJT_polyhedral} that $\kar(A_0)$ is trivial.
\end{remark}

With Lemma \ref{lem:clr}, we can now deduce well-posedness of the variational problem to be studied. We denote {$\mathcal{A}_0\coloneqq A_0 \cap (\kar(A_0)^\bot \oplus \rge(A_0))$ the restriction to $\kar(A_0)^\bot$ and adstriction to $\rge(A_0)$ of $A_0$, in other words
\[
   \mathcal{A}_0 \colon \dom(A_0)\cap \kar(A_0)^\bot \subseteq \kar(A_0)^\bot \to \rge(A_0), x\mapsto A_0x.
\]
}

\begin{theorem}\label{thm:wpimp} Let $\Omega$ have the impedance compactness property, $a\in \mathcal{B}(L^2(\Omega)^8)$. Assume that $\Re \langle aq,q\rangle\geq \alpha\langle q,q\rangle$ for all $q\in \rge(A_0)$. Then for all $F\in \dom(\mathcal{A}_0)^*$ there exists a unique $(E,H)\in \dom(\mathcal{A}_0)$ such that for all $(E',H')\in \dom(\mathcal{A}_0)$ we have
\[
    \langle aA_0(E,H),A_0 (E',H')\rangle = F((E',H')).
\]
Moreover, we have
\[
    (E,H) = \mathcal{A}_0^{-1} (\iota^* a \iota)^{-1} (\mathcal{A}_0^\diamond)^{-1} F,
\]
where $\iota\colon \rge(A_0)\hookrightarrow L^2(\Omega)^8$ is the canonical embedding and 
\[
   \mathcal{A}_0^\diamond \colon \rge(A_0) \to \dom(\mathcal{A}_0)^*, q\mapsto ((E,H)\mapsto \langle q,A_0(E,H)\rangle).
\]
\end{theorem}
\begin{proof} The assertion follows from \cite[Theorem 2.9]{W18} (see also \cite[Theorem 3.1]{TW14} for a more general version) applied to $C=A_0$; using that $\rge(A_0)$ is closed by Lemma \ref{lem:clr}.
\end{proof}

\begin{remark} The statement in Theorem \ref{thm:wpimp} follows from the classical Lax--Milgram lemma (disregarding the formula for $(E,H)$). Indeed, the necessary coerciveness condition in the Lax--Milgram lemma is implied by the inequality given in Lemma \ref{lem:clr}.
\end{remark}

\subsection{The convergence statement}

Throughout this section, we shall assume that $\Omega$ has Lipschitz boundary and the impedance compactness property, which yields that $A_0$ has closed range, by Lemma \ref{lem:clr}.
There are now several possibilities to address convergence of coefficient sequences of the variational problem associated with $A_0$. We choose to start by defining a notion analogous to the above notion of nonlocal $H$-convergence. For this, we need to slightly adapt this notion and the set of admissible coefficients. For $0<\alpha\leq\beta$ we set
\begin{multline*}
     \mathcal{M}(\alpha,\beta,A_0)\coloneqq \{ a\in \mathcal{B}(L^2(\Omega)^8); \\ \Re a_{00}, \Re a_{11}-a_{10}a_{00}^{-1}a_{01}\geq \alpha,  \Re a_{00}^{-1}, \Re (a_{11}-a_{10}a_{00}^{-1}a_{01})^{-1}\geq 1/\beta\},
\end{multline*}
where, here, we used the notation $c_{ij} = \iota_i^* c \iota_j$ with $\iota_i = \begin{cases} \rge(A_0)\hookrightarrow L^2(\Omega)^8, & i=0, \\
 \rge(A_0)^\bot \hookrightarrow L^2(\Omega)^8, &i=1,\end{cases}$ for $i,j\in \{0,1\}$.
 
 The adapted nonlocal $H$-convergence, now reads as follows. We note the similarity to the nonlocal $H$-convergence introduced above. In particular, in the light of Theorem \ref{thm:NHC_top}.
 
 \begin{definition}A sequence of coefficients $(c_n)_n$ in $ \mathcal{M}(\alpha,\beta,A_0)$ \emph{nonlocally $H$-converges w.r.t. $A_0$} to some $c\in \mathcal{B}(L^2(\Omega)^8)$, if 
 \begin{align*}
     c_{n,00}^{-1}&  \to c_{00}^{-1}, \\
     c_{n,00}^{-1}c_{n,01} & \to      c_{00}^{-1}c_{01}, \\
     c_{n,10}c_{n,00}^{-1} & \to      c_{10}c_{00}^{-1}, \\
     c_{n,11}-c_{n,10} c_{n,00}^{-1}c_{n,01} & \to      c_{11}-c_{10} c_{00}^{-1}c_{01},
 \end{align*}in the respective weak operator topologies; we denote the topology induced on $ \mathcal{M}(\alpha,\beta,A_0)$ by  $\tau_{\textnormal{nlH},A_0}$; see also Theorem \ref{thm:NHC_top}.
 \end{definition} Using standard estimates for weakly convergent sequences, we infer $c\in  \mathcal{M}(\alpha,\beta,A_0)$; see also \cite[Lemma 2.12]{W18}.
 
 A straightforward application of the properties inherited by the introduced convergence is the following.
 
 \begin{proposition}\label{prop:conv1} Let $(a_n)_n$ be in $\mathcal{M}(\alpha,\beta,A_0)$ nonlocally $H$-converging w.r.t. $A_0$ to some $a$. Then, for all $F\in\dom(\mathcal{A}_0)^*$ and $(E_n,H_n)\in\dom(\mathcal{A}_0)$ being the solution of
 \[
\langle a_n A_0(E_n,H_n), A_0(E',H')\rangle = F(E',H')\quad ((E',H')\in \dom(\mathcal{A}_0)),
 \]we have $(E_n,H_n)\rightharpoonup (E,H)$ in $\dom(\mathcal{A}_0)$ and $a_n A_0(E_n,H_n) \rightharpoonup  aA_0(E,H)\in L^2(\Omega)^8$, where $(E,H)$ satisfy
 \[
    \langle aA_0(E,H), A_0(E',H')\rangle = F(E',H')\quad ((E',H')\in \dom(\mathcal{A}_0)).
 \] 
 \end{proposition}
 \begin{proof}
 We use the solution formula provided by Theorem \ref{thm:wpimp}. Then using $\iota_0\colon \rge(A_0)\hookrightarrow L^2(\Omega)^8$, we deduce for all $n\in \N$
 \[
   ( E_n,H_n) = \mathcal{A}_0^{-1} ( \iota_0^* a_n\iota_0)^{-1} (\mathcal{A}_0^\diamond)^{-1} F =  \mathcal{A}_0^{-1} a_{n,00}^{-1} (\mathcal{A}_0^\diamond)^{-1} F \rightharpoonup  \mathcal{A}_0^{-1} a_{00}^{-1} (\mathcal{A}_0^\diamond)^{-1} F
 \] as $n\to\infty$. Furthermore, we have for $n\in\N$ with Theorem \ref{thm:wpimp} again (and $\iota_1 \colon \rge(A_0)^\bot \hookrightarrow L^2(\Omega)^8$)
 \begin{align*}
   a_n A_0(E_n,H_n) & = a_n \iota_0\mathcal{A}_0 \mathcal{A}_0^{-1} a_{n,00}^{-1} (\mathcal{A}_0^\diamond)^{-1} F
   \\ & = \begin{pmatrix} \iota_0 & \iota_1 \end{pmatrix} \begin{pmatrix} \iota_0^* \\ \iota_1^* \end{pmatrix}a_n \iota_0 a_{n,00}^{-1} (\mathcal{A}_0^\diamond)^{-1} F
      \\ & = \begin{pmatrix} \iota_0 & \iota_1 \end{pmatrix} \begin{pmatrix}a_{n,00} \\ a_{n,10} \end{pmatrix} a_{n,00}^{-1} (\mathcal{A}_0^\diamond)^{-1} F
            \\ & = \begin{pmatrix} \iota_0 & \iota_1 \end{pmatrix} \begin{pmatrix}(\mathcal{A}_0^\diamond)^{-1} F \\ a_{n,10} a_{n,00}^{-1} (\mathcal{A}_0^\diamond)^{-1} F\end{pmatrix} 
            \\ & \rightharpoonup \begin{pmatrix} \iota_0 & \iota_1 \end{pmatrix} \begin{pmatrix}(\mathcal{A}_0^\diamond)^{-1} F \\ a_{10} a_{00}^{-1} (\mathcal{A}_0^\diamond)^{-1} F\end{pmatrix} 
            \\ & = a A_0(E,H),
 \end{align*}
 as $n\to\infty$.
  \end{proof}
  
 \subsection{Local blockdiagonal coefficients and local $H$-convergence}

In this section, we shall consider more specifically coefficients of the following form
\[
    \diag(a_e,b_e,a_h,b_h) = \begin{pmatrix} a_e & 0 & 0 & 0 \\ 0 & b_e & 0 & 0 \\ 0 & 0 & a_h & 0 \\ 0 & 0 & 0 & b_h \end{pmatrix} = a \in  \mathcal{M}(\alpha,\beta,A_0)
\]for some $ a_{e}, a_{h}\in M(\alpha,\beta,\Omega)$ and $ b_{e}, b_{h}\in L^\infty(\Omega))$ with $\Re b_e, \Re b_h\geq \alpha$ and $\Re  b_e^{-1}, \Re b_h^{-1}\geq 1/\beta$. We shall denote this subset of $\mathcal{M}(\alpha,\beta,A_0)$ by ${M}_{\diag}(\alpha,\beta,A_0)$. 

The main theorem of this section is about the invariance of ${M}_{\diag}(\alpha,\beta,A_0)$ under nonlocal $H$-convergence w.r.t.~$A_0$ and about an explicit description for this convergence in terms of local topologies.

 \begin{theorem}\label{thm:diag} Let $(a_n)_n = (\diag( a_{e,n}, b_{e,n}, a_{h,n}, b_{h,n} ))_n$  in ${M}_{\diag}(\alpha,\beta,A_0)$ and assume $a\in \mathcal{M}(\alpha,\beta,A_0)\cap L^\infty(\Omega)^{8\times 8}$. Then the following conditions are equivalent:
 \begin{enumerate}
   \item[(i)] $(a_n)_n\to a$ $H$-nonlocally w.r.t~$A_0$;
   \item[(ii)] $a =   \diag(a_e,b_e,a_h,b_h) \in {M}_{\diag}(\alpha,\beta,A_0)$ and as $n\to\infty$
   \begin{align*}
      & a_{e,n}^{-1} \to a_{e}^{-1} \quad  a_{h,n}^{-1} \to a_{h}^{-1} \text{ $H$-locally as $n\to\infty$}\\
            & b_{e,n}^{-1} \to b_{e}^{-1} \quad b_{h,n}^{-1} \to b_{h}^{-1} \text{ in $\sigma(L^\infty(\Omega),L^1(\Omega))$ as $n\to\infty$}\\
   \end{align*}
 \end{enumerate}
\end{theorem}
For the proof of this theorem, we invoke and recall the subsequent compactness statements.

Using the compactness of closed bounded sets of linear operators under the weak operator topology (see e.g.~Theorem 5.1), we deduce the next theorem (in order to render $(\mathcal{B},\tau_{\textnormal{nlH},A_0})$ a Hausdorff space, one should consult the (easy) argument  \cite[Proposition 5.4]{W18}).
\begin{theorem}[{{see also \cite[Theorem 5.10, Theorem 5.5, Proposition 5.4]{W18}}}]\label{thm:comp1} Let $\mathcal{B}\subseteq  \mathcal{M}(\alpha,\beta,A_0)$ be bounded in $\mathcal{B}(L^2(\Omega)^8)$. Then $(\mathcal{B},\tau_{\textnormal{nlH},A_0})$ is a relatively compact Hausdorff space; its closure is metrisable and sequentially compact.
\end{theorem}

\begin{theorem}[{{\cite[Theorem 6.5 \& the argument after Definition 6.4]{Tar09} and  Banach--Alaoglu theorem}}]\label{thm:comp2} The space $ M(\alpha,\beta,\Omega)$ is compact (and metrisable) under local $H$-convergence.
{Furthermore bounded} sets of $L^\infty(\Omega)$ are relatively compact (and metrisable) under  $\sigma(L^\infty,L^1)$.
\end{theorem}

\begin{proof}[Proof of Theorem \ref{thm:diag}]
 By the Theorems \ref{thm:comp1} and \ref{thm:comp2}, appealing to the subsequence principle, it suffices to show the following: Assume that $(a_n)_n\to a$ $H$-nonlocally w.r.t~$A_0$ and that    \begin{align*}
      & a_{e,n}^{-1} \to a_{e}^{-1} \quad  a_{h}^{-1} \to a_{h}^{-1} \text{ $H$-locally as $n\to\infty$}\\
            & b_{e,n}^{-1} \to b_{e}^{-1} \quad b_{h,n}^{-1} \to b_{h}^{-1} \text{ in $\sigma(L^\infty(\Omega),L^1(\Omega))$ as $n\to\infty$}\\
   \end{align*} for some $ a_{e}, a_{h}\in M(\alpha,\beta,\Omega)$ and $ b_{e}, b_{h}\in L^\infty(\Omega))$ with $\Re b_e, \Re b_h\geq \alpha$ and $\Re  b_e^{-1}, \Re b_h^{-1}\geq 1/\beta$. Then $a= \diag(a_e,b_e,a_h,b_h) $.
   As the argument  for $a_e$ is similar to $a_h,b_h,b_e$, we only focus on $a_e$ first. Let $E \in \mathcal{H}_0(\curl)$, $E\neq 0$.  Then $\dive E=0$. Moreover $(E,0)\in \dom(A_0)$ and $A_0(E,0)=(\curl E,0,0,0)^\top$. Define $F\in \dom(\mathcal{A}_0)^*$ by
   \begin{multline*}
        F(E',H')=\langle a_e \curl E, \curl E'\rangle = \langle \diag(a_e,b_e,a_h,b_h) A_0(E,0),A_0(E',H')\rangle \\ ((E',H')\in \dom(A_0)).
   \end{multline*}
  Let now $(E_n,H_n)\in\dom(\mathcal{A}_0)$ be the solution of
  \begin{multline*}
      \langle a_{n} A_0(E_n,H_n),A_0(E',H')\rangle= \langle \diag(a_e,b_e,a_h,b_h) A_0(E,0),A_0(E',H')\rangle \\ ((E',H')\in \dom(A_0)).
  \end{multline*}
  Then, by nonlocal $H$-convergence of $(a_n)_n$, we deduce $(E_n,H_n)\rightharpoonup (E,0)$ in $\dom(\mathcal{A}_0)$ and $a_n A_0(E_n,H_n)\to a A_0 (E,0)$ and by the unique solvability of the limit problem, we obtain
      \[
      \langle a A_0(E,0),A_0(E',H')\rangle= \langle \diag(a_e,b_e,a_h,b_h) A_0(E,0),A_0(E',H')\rangle \quad ((E',H')\in \dom(A_0)).
      \]
      On the other hand, let $\tilde{E}_n\in \mathcal{H}_0(\curl)$ be the unique solution of
      \[
          \langle a_{e,n} \curl \tilde{E}_n, \curl E'\rangle =          \langle a_{e} \curl {E}, \curl E'\rangle \quad(E'\in\mathcal{H}_0(\curl)).
      \]
We then see  that $(\tilde{E}_n,0)$ satisfies the same equation as $(E_n,H_n)$ does; thus $E_n=\tilde{E}_n$ and $H_n=0$ for all $n\in \N$. Moreover, we deduce by local $H$-convergence of $a_{e,n}^{-1}\to a_e^{-1}$, that
     $a_{e,n}\curl E_n \rightharpoonup a_{e}\curl E$\red{; see also Remark \ref{rem:anEn} applied to $v_n=E_n$}. Since, $a_n A_0(E_n,H_n)=(a_{e,n}\curl E_n,0,0,0)^\top$ and $a A_0 (E,0)= (a_{e}\curl E,0,0,0)^\top$, we infer using $a\in L^\infty(\Omega)^{8\times 8}$ and Proposition \ref{prop:ws_ident} (see also Remark \ref{rem:ident}) that $a_e = (a_{ij})_{i,j\in\{1,2,3\}}$.
     By explicitly constructing right-hand sides (it suffices to use $E$ and $H$ being compactly supported $C_c^\infty$-vector fields), we proceed to eventually obtain that the diagonal entries of $a$ coincide with the ones of $\diag(a_e,b_e,a_h,b_h)$.   
     Next, we treat the off-diagonal components. We use the representation
     \[
         a = \begin{pmatrix} a^{11} & a^{12} & a^{13} & a^{14} \\ a^{21} & a^{22} & a^{23} & a^{24} \\ a^{31} & a^{32} & a^{33} & a^{34}\\  a^{41} & a^{42} & a^{43} & a^{44} \end{pmatrix} \in \mathcal{B}\big(L^2(\Omega)^3\oplus   L^2(\Omega) \oplus L^2(\Omega)^3\oplus   L^2(\Omega)\big).
     \]
     We need to show that $a^{ij}=0$ provided $i,j\in \{1,2,3,4\},i\neq j$. For this, we assume that $ \begin{pmatrix} 0 & a^{12} & a^{13} & a^{14} \end{pmatrix} \neq 0$. Next, if $a^{12}\neq 0$ (the other cases can be dealt with similarly), take a $C_c^\infty(\Omega)$-function $\phi$ supported on a small ball. $\phi$ is not constant, $\grad \phi\neq 0$ and $\phi$ is not harmonic on its support since $\phi$ does not satisfy the maximum principle. We have $\dive\grad \phi \neq 0$ and by appropriately shifting $\phi$, we may choose $\phi$ so that $a^{12}\dive\grad \phi \neq 0$ (we shall further specify $\phi$ later). We set $E\coloneqq \grad \phi$ and define $F\in \dom(\mathcal{A}_0)^*$ by (note that $(E,0)\in \dom({A}_0)$)
   \begin{align*}
        F(E',H') & = \langle a A_0(E,0),A_0(E',H')\rangle \\
       &  = \langle  \begin{pmatrix} a^{11} & a^{12} & a^{13} & a^{14} \\ a^{21} & a^{22} & a^{23} & a^{24} \\ a^{31} & a^{32} & a^{33} & a^{34}\\  a^{41} & a^{42} & a^{43} & a^{44} \end{pmatrix}
        \begin{pmatrix} \curl E\\ \dive E \\ \curl 0 \\  \dive 0 \end{pmatrix}  ,        \begin{pmatrix} \curl E'\\ \dive E' \\ \curl H' \\  \dive H' \end{pmatrix}\rangle \\
              &  = \langle  \begin{pmatrix} a^{11} & a^{12} & a^{13} & a^{14} \\ a^{21} & a^{22} & a^{23} & a^{24} \\ a^{31} & a^{32} & a^{33} & a^{34}\\  a^{41} & a^{42} & a^{43} & a^{44} \end{pmatrix}
        \begin{pmatrix} \curl \grad \phi \\ \dive \grad \phi \\   0 \\   0 \end{pmatrix}  ,        \begin{pmatrix} \curl E'\\ \dive E' \\ \curl H' \\  \dive H' \end{pmatrix}\rangle \\
  &                 = \langle  \begin{pmatrix} a^{12}\dive \grad \phi  \\   a^{22}\dive \grad \phi  \\   a^{32}\dive \grad \phi \\    a^{42}\dive \grad \phi \end{pmatrix}
      ,        \begin{pmatrix} \curl E'\\ \dive E' \\ \curl H' \\  \dive H' \end{pmatrix}\rangle     \\
&      = \langle   a^{(2)}\dive \grad \phi   ,        \begin{pmatrix} \curl E'\\ \dive E' \\ \curl H' \\  \dive H' \end{pmatrix}\rangle        \quad ((E',H')\in \dom(A_0)),
   \end{align*}
   where $a^{(2)}$ is the second column of $a$ along the above block decomposition.
    Let now $(E_n,H_n)\in\dom(\mathcal{A}_0)$ be the solution of
  \[
      \langle a_{n} A_0(E_n,H_n),A_0(E',H')\rangle=F(E',H') \quad ((E',H')\in \dom(A_0)).
  \]
  By nonlocal $H$-convergence and Theorem  \ref{thm:wpimp}, we deduce 
  \begin{equation}\label{eq:1}
      a_{n} A_0(E_n,H_n) = \begin{pmatrix} a_{e,n}\curl E_n \\ b_{e,n} \dive E_n \\ a_{h,n} \curl H_n \\ b_{h,n} \dive H_n \end{pmatrix} \rightharpoonup a A_0(E,0) = a^{(2)}\dive \grad \phi
  \end{equation}
  and $(E_n,H_n)\rightharpoonup (E,0)$ in $\dom(A_0)$. Also, we have for $E'\in \curl [C_c^\infty(\Omega)^3]$
  \begin{equation}\label{eq:aenEn}
      \langle a_{e,n}\curl E_n, \curl E' \rangle = \langle a^{12}\dive\grad \phi, \curl E'\rangle.
  \end{equation}
  \red{We intend to let $n\to\infty$ in this equation. Since, in general, $E_n\notin \mathcal{H}_0(\curl)$ we cannot simply use the convergence of `fluxes' as stated in Remark \ref{rem:anEn}. We employ the local character of local $H$-convergence, which is why we invoke Lemma \ref{lem:dcl} below. For this we read off from \eqref{eq:aenEn} that}
  $\curl a_{e,n}\curl E_n$ is constant and, thus, relatively compact in $\mathcal{H}_0(\curl)^*$.
 By Lemma \ref{lem:dcl} we deduce from local $H$-convergence of $a_{e,n}^{-1}$, that
 \[
    \langle a_{e,n}\curl E_n, \curl E' \rangle \to  \langle a_{e}\curl E, \curl E' \rangle \quad(E'\in \curl [C_c^\infty(\Omega)^3])
 \]
 On the other hand, we have that $E= \grad \phi $ and so $a_e\curl E = a_e \curl \grad \phi=0$. Hence,
  \[
\langle a^{12} \dive\grad \phi, \curl E'\rangle=0,  ( E'\in \curl [C_c^\infty(\Omega)^3]).
\]
With the choice of $\phi$ provided in Lemma \ref{lem:tesf} (applied to $a^{12}\in L^\infty(\Omega)^3\subseteq L^2(\Omega)^3$ in place of $a$), we infer a contradiction, since locally $a^{12}\dive\grad \phi $ is not orthogonal to the range of $\curl$ with homogeneous boundary conditions.
  \end{proof}
  
  \begin{lemma}[{{{see \cite[Lemma 10.3]{Tar09}}}}]\label{lem:dcl} Assume that $(a_n^{-1})_n$ in $M(\alpha,\beta,\Omega)$ locally $H$-converges to some $a^{-1}$ and $E_n\rightharpoonup E$ in $L^2(\Omega)^3$ and $\curl E_n \rightharpoonup \curl E$ in $L^2(\Omega)^3$ and assume that $\curl a_n \curl E_n$ belongs to a compact subset of $\mathcal{H}_0(\curl)^*$.
  
   Then $ \langle a_n \curl E_n,\curl E'\rangle \to \langle a \curl E,\curl E'\rangle$ for all $E'\in C_c^\infty(\Omega)^3$.
  \end{lemma}
  \begin{proof}
   The statement follows after appropriately adapting  \cite[Lemma 10.3]{Tar09} for the $\curl$-case and using the classical $\dive$-$\curl$-lemma.
  \end{proof}
 
\begin{lemma}\label{lem:tesf} Let $a\in L^2(B_{\R^3}(0,1))^3$, $a\neq 0$. Then there exists $\phi \in C_c^\infty(B(0,1))$ such that $(\dive\grad \phi)  a$ is not in the range of $\grad$.
\end{lemma}
\begin{proof}
We distinguish two cases. We start with $a=\grad\eta$ for some $\eta \in H^1(B(0,1))$. Let $\eta_1 \in C_c^\infty(B(0,1))$ be such that $\grad \eta_1 \times \grad \eta \neq 0$. We find $\phi_1\in H_0^1(B(0,1))$ satisfying $\dive\grad \phi_1 =\eta_1$. We define $\phi$ to be $\phi_1$ multiplied with a smooth cut-off, being $1$ on the support of $\eta_1$. Then, by standard regularity theory for the Laplacian, we obtain $\phi\in C_c^\infty(B(0,1))$ and letting $\psi \coloneqq\dive\grad \phi$ we get
\[
    \curl(\psi\grad \eta) = \grad\psi \times \grad \eta,
\]
which is a non-zero function on the support of $a=\grad\eta$; thus $\psi a$ is not a gradient.

If $a$ is not a gradient of an $H^1$-function, its distributional $\curl$ is non-zero. In this case, we find $x_0\in  B(0,\delta)$ and $\delta\in (0,1)$ so that $\overline{B(x_0,\delta)}\subset B(0,1)$ and $\spt \curl a \cap B(x_0,\delta)\neq \emptyset$. Let $\chi\in C_c^\infty(B(0,1))$ be $1$ on $B(x_0,\delta)$. Define $\phi_1\in H_0^1(\Omega)$ to be the solution of $\dive\grad\phi_1 = \chi$. Define $\phi\coloneqq \chi\phi_1$ and $\psi\coloneqq\dive\grad\phi$. Then we obtain on $B(x_0,\delta)$
\[
    \curl(\psi a) = \psi \curl a + \grad \psi \times a = \psi\curl a\neq 0,
\]which yields that $\psi a$ is not a gradient.
\end{proof}
 
We have thus provided a characterisation for the nonlocal $H$-convergence for coefficients of the curl-div system in terms of nonlocal $H$-convergence (and even local $H$-convergence) of the diagonal entries. These results are relevant for the scattering problem in the context of Maxwell's equations. Of course it would be desirable to have a similar result also for nonlocal coefficients. Given that the topology induced by nonlocal $H$-convergence depends on the boundary conditions involved (see \cite{W18}), a result along the lines of Theorem \ref{thm:diag} for nonlocal coefficients instead of local ones is unlikely to be true without additional assumptions on the coefficients. We will postpone an analysis of this issue to future work.
 
 \section*{Acknowledgements}
 
 We thank the anonymous referee for their insightful comments helping to significantly improve the paper.

\end{document}